\documentclass[12pt]{amsart}
 \usepackage{amsmath, amsthm, amsfonts, amssymb, color}
\usepackage{amsfonts, amsmath}
\begin{document}
 \baselineskip 18pt
\hfuzz=6pt

\newtheorem{theorem}{Theorem}[section]
\newtheorem{prop}[theorem]{Proposition}
\newtheorem{lemma}[theorem]{Lemma}
\newtheorem{definition}[theorem]{Definition}
\newtheorem{cor}[theorem]{Corollary}
\newtheorem{example}[theorem]{Example}
\newtheorem{remark}[theorem]{Remark}
\newcommand{\ra}{\rightarrow}
\renewcommand{\theequation}
{\thesection.\arabic{equation}}
\newcommand{\ccc}{{\mathcal C}}
\newcommand{\one}{1\hspace{-4.5pt}1}

\def \Gg {\widetilde{{\mathcal G}}_{L}}
\def \GG {{\mathcal G}_{L}}
\def \SL {\sqrt{L}}
\def \GL{G_{\mu,\Phi}^{\ast}}
\def \gL{g_{\mu,\Psi}^{\ast}}
\def \RN {\mathbb{R}^n}
\def\RR{\mathbb R}
\newcommand{\nf}{\infty}

\allowdisplaybreaks

\title[Weighted $L^p$ estimates  for the   area integral]
{ Weighted $L^p$ estimates  for the  area integral\\ [2pt] associated to   self-adjoint operators}

\author{Ruming Gong   \   and \   Lixin Yan}
 \footnotetext[1]{{\it {\rm 2000} Mathematics Subject Classification:}
42B20, 42B25, 46E35, 47F05.}
\footnotetext[1]{{\it Key words and phrase:} Weighted norm inequalities, area integral, self-adjoint operators, heat kernel,
semigroup, Whitney decomposition, $A_p$ weighs.}
\address{
Ruming Gong, Department of Mathematics, Sun Yat-sen (Zhongshan) University, Guangzhou, 510275, P.R. China}
\email{
gongruming@163.com}
\address{
Lixin Yan, Department of Mathematics, Sun Yat-sen (Zhongshan) University, Guangzhou, 510275, P.R. China}
\email{
mcsylx@mail.sysu.edu.cn
}


\begin{abstract} This article is concerned with  some weighted norm inequalities 
 for the so-called horizontal (i.e. involving time derivatives) area integrals  associated to  a non-negative self-adjoint operator
satisfying  a pointwise Gaussian estimate  for its heat kernel, 
as well as     the corresponding  vertical (i.e. involving space derivatives) area integrals  associated to  a non-negative self-adjoint operator
satisfying in addition a pointwise upper bounds for the gradient of the heat  kernel. 
As applications, we  obtain   sharp estimates for the operator norm of  the area integrals on $L^p(\RN)$ as $p$ becomes large,
and the growth of the $A_p$ constant on estimates  of  the area integrals 
 on the weighted $L^p$ spaces.
\end{abstract}

\maketitle

  \tableofcontents

\section{Introduction  }
\setcounter{equation}{0}

\noindent
{\bf 1.1. Background.}\
Let $\varphi\in C_0^{\infty}({\RN})$ with $\int \varphi =0.$
Let $\varphi_t(x)=t^{-n}\varphi(x/t), t>0$, and define the Lusin area integral by

\begin{eqnarray}\label{e1.1}
S_{\varphi}(f)(x)=\bigg(\int_{|x-y|<t}
\big|f\ast \varphi_t(y)\big|^2 {dy \, dt\over t^{n+1}}\bigg)^{1/2}.
\end{eqnarray}

\smallskip

\noindent
A celebrated result of
Chang-Wilson-Wolff (\cite{CWW}) says that for all $w\geq 0$, $w\in L_{\rm loc}^{1}(\RN)$ and all $f\in \mathcal{S}(\RN)$,
  there is a constant $C=C(n, \varphi)$ independent of $w$ and $f$ such that

\begin{eqnarray}\label{e1.2}
\int_{\RN}S^2_{\varphi}(f)w\, dx\leq C\int_{\RN}|f|^2Mw\, dx,
\end{eqnarray}

\noindent
where  $Mw$ denotes the Hardy-Littlewood maximal operator of $w$.

\smallskip

The fact that $\varphi$ has compact support is crucial  in the proof of Chang, Wilson and Wolff.
In \cite{CW},  Chanillo and Wheeden  overcame this difficulty, and they    obtained  weighted $L^p$ inequalities for
$1<p<\infty$ of
the area integral, even when $\varphi$
does not have compact support, including the classical   area function  defined by means of the Poisson kernel.




\smallskip

From   the theorem  of Chang, Wilson and Wolff, it was already observed in \cite{FP}   that R. Fefferman and Pipher obtained
 sharp estimates for the operator norm of    a classical Calder\'on-Zygmund
singular integral, or the classical   area integral for $p$ tending to infinity, e.g.,

\begin{eqnarray}\label{e1.3}
\big\|S_{\varphi}(f)\big\|_{L^p(\RN)}\leq C p^{1/2} \big\|f\big\|_{L^p(\RN)}
\end{eqnarray}

 \noindent
 as $p\rightarrow \infty.$\\

\noindent
{\bf 1.2. Assumptions, notation and definitions.}\, In this article, our main goal is
to provide an extension of the result  of 
Chang-Wilson-Wolff  to study  some weighted norm inequalities 
  for the area integrals  associated to    non-negative self-adjoint operators,
 whose kernels are not smooth enough
to fall under the scope of \cite{CWW, CW, W}.
The relevant classes of operators is determined by the following condition: 

\smallskip

\noindent
{\bf Assumption $(H_1)$.}\, 
Assume that  $L$ is a  non-negative self-adjoint operator  on $L^2({\mathbb R}^{n}),$
 the semigroup $e^{-tL}$, generated by $-L$ on $L^2(\RN)$,  has the kernel  $p_t(x,y)$
which  satisfies
the following  Gaussian upper bound if there exist  $C$ and $c$ such that
 for all $x,y\in {\mathbb R}^{n}, t>0,$

 $$
|p_{t}(x,y)| \leq \frac{C}{t^{n/2} } \exp\Big(-{|x-y|^2\over
c\,t}\Big).
\leqno{(GE)}
$$

\smallskip

\noindent

Such estimates are typical for elliptic or sub-elliptic differential operators of second
order (see for instance, \cite{Da} and \cite{DOS}).\\  

 For $f\in {\mathcal S}(\RN)$, define  the (so called vertical) area functions  $S_{P}$ and $S_{H}$
 by

 \begin{eqnarray}\label{e1.4}
S_{P}f(x)&=&\bigg(\int_{|x-y|<t}
|t\nabla_y e^{-t\sqrt{L}} f(y)|^2 {dy dt\over t^{n+1}}\bigg)^{1/2},\\
S_{H}f(x)&=&\bigg(\int_{|x-y|<t}
|t\nabla_y e^{-t^2L} f(y)|^2 {dy dt\over t^{n+1}}\bigg)^{1/2}, \label{e1.5}
\end{eqnarray}

\noindent
as well as the (so-called horizontal) area functions $s_{p}$ and $s_{h}$
 by

 \begin{eqnarray}\label{e1.6}
s_{p}f(x)&=&\bigg(\int_{|x-y|<t}
|t\sqrt{L} e^{-t\sqrt{L}} f(y)|^2 {dy dt\over t^{n+1}}\bigg)^{1/2},\\
s_{h}f(x)&=&\bigg(\int_{|x-y|<t}
|t^2L e^{-t^2L} f(y)|^2 {dy dt\over t^{n+1}}\bigg)^{1/2}.\label{e1.7}
\end{eqnarray}

\smallskip

It is well known (cf. e.g. \cite{St, G}) that  when    $L=\Delta$ is the Laplacian on $\RN$, the classical 
area functions $S_{P}, S_{H}, s_{p}$ and $s_{h}$  are all bounded on $L^p(\RN), 1<p<\infty.$
For a general non-negative self-adjoint  operator $L$, $L^p$-boundedness of 
  the area functions $S_{P}, S_{H}, s_{p}$ and $s_{h}$ associated to $L$  has  been 
studied extensively -- see for examples   \cite{A}, \cite{ACDH}, \cite{ADM}, \cite{CDL}, \cite{St1} and \cite{Y}, and the references therein.   
 
 \medskip

\noindent
{\bf 1.3. Statement of the main results.}
Firstly, we have the following weighted $L^p$ estimates for the area functions
 $s_{p}$ and $s_{h}.$

 \medskip

\begin{theorem}\label{th1.1} \    Let $L$ be a non-negative self-adjoint operator such that the corresponding
 heat kernels satisfy Gaussian bounds $(GE)$.    
If $w\geq 0$, $w\in L_{\rm loc}^{1}(\RN)$  and  $f\in \mathcal{S}(\RN)$, then

\smallskip
\begin{eqnarray*}
\hspace{-2.5cm}
&{\rm (a)}& \hspace{0.1cm}
\int_{\RN}s_h(f)^pw\, dx\leq C(n,p)\int_{\RN}|f|^pMw\, dx,\ \ \ 1<p\leq 2,\\
\hspace{-2.5cm}&{\rm (b)}& \hspace{0.1cm}\int_{\{s_h(f)>\lambda\}} wdx\leq {C(n)\over \lambda}\int_{\RN}|f| Mwdx,\ \ \ \lambda>0,
\\
\hspace{-2.5cm}&{\rm (c)}& \hspace{0.1cm} \int_{\RN} s_h(f)^pwdx\leq C(n,p)\int_{\RN}|f|^p(Mw)^{p/2}w^{-(p/2-1)}dx, \ \ \ 2<p<\infty.
\end{eqnarray*}

\smallskip
Also, estimates (a), (b) and (c) hold  for the operator $s_p.$
 \end{theorem}

\medskip

To study  weighted $L^p$-boundedness of   the (so-called vertical) area integrals $S_{P}$ and $S_{H}$,   
one assumes in addition   the following condition:

\smallskip

\noindent
{\bf Assumption $(H_2)$.}\, Assume that 
the semigroup $e^{-tL}$, generated by $-L$ on $L^2(\RN)$,  has the kernel  $p_t(x,y)$
which  satisfies a pointwise upper bound for the gradient of  the heat kernel. That is, there exist  $C$ and $c$  such that for all $x,y\in{\RN}, t>0$,

$$
\big|\nabla_x p_t(x,y)\big|\leq {C\over t^{(n+1)/2}} \exp\Big(-{|x-y|^2\over
c\,t}\Big).
\leqno{(G)}
$$

\medskip

Then the following result holds.

\medskip

 \begin{theorem}\label{th1.2} \    Let $L$ be a non-negative self-adjoint operator such that the corresponding
 heat kernels satisfy conditions  $(GE)$ and $(G)$.
 \noindent
 If $w\geq 0$, $w\in L_{\rm loc}^{1}(\RN)$  and  $f\in \mathcal{S}(\RN)$, then $(a), (b)$ and $(c)$ of Theorem~\ref{th1.1}
 hold for  the area functions $S_{P}$ and  $S_{H}$.
 \end{theorem}

\smallskip

Let us now  recall a definition. We say that a weight $w$ is in the  the Muckenhoupt class $A_p, 1<p<\infty,$ if

\begin{eqnarray*}
\|w\|_{A_p}\equiv \sup_Q \bigg({1\over |Q|}\int_Q w(x)dx \bigg)\bigg({1\over |Q|}\int_Q w(x)^{-1/(p-1)}dx \bigg)^{p-1}<\infty.
\end{eqnarray*}

\noindent
$\|w\|_{A_p}$ is usually called the $A_p$ constant  (or characterization or norm) of the weight. The case $p=1$
is understand by replacing the right hand side by $(\inf_Q w)^{-1}$ which is equivalent to the one defined above.
Observe the duality relation:

$$
\|w\|_{A_p}=\|w^{1-p'}\|^{p-1}_{A_{p'}}.
$$

\smallskip

Following the  R. Fefferman-Pipher's method, we can use  Theorems~\ref{th1.1} and ~\ref{th1.2}  to
establish the   $L^p$ estimates of the area integrals as $p$ becomes large.

 \medskip

 \begin{theorem}\label{th1.3}  Let $T$ be  of the area functions $s_h$, $s_p$, $S_{P}$ and $S_{H}.$
 Under assumptions of  Theorems~\ref{th1.1} and ~\ref{th1.2},  there exists a constant $C$ such that
 for all $w\in A_1$, the following  
estimate holds:

\begin{eqnarray}\label{e1.8}
\|Tf\|_{L^2_w(\RN)}\leq C\|w\|_{A_1}^{1/2}\|f\|_{L^2_w(\RN)}.
\end{eqnarray}

\noindent This inequality   implies  that as $p\rightarrow \infty$

\begin{eqnarray}\label{e1.9}
\|Tf\|_{L^p(\RN)}\leq Cp^{1/2}\|f\|_{L^p(\RN)}.
\end{eqnarray}
\end{theorem}

\bigskip

 The next result  we will prove is the following.

\begin{theorem}\label{th1.4}  Let $T$ be  of the area functions $s_h$, $s_p$, $S_{P}$ and $S_{H}.$
 Under assumptions of  Theorems~\ref{th1.1} and ~\ref{th1.2},  there exists a constant $C$ such that
 for all $w\in A_p$,  the following
estimate holds for all $f\in L^p_w(\RN), 1<p<\infty$:

\begin{eqnarray}\label{e1.10}
\|T f\|_{L^p_w(\RN)}\leq C\|w\|_{A_p}^{\beta_p+1/(p-1)} \|f\|_{L^p_w(\RN)} \ \ (1<p<\infty),
\end{eqnarray}

\noindent where $\beta_p=\max\{1/2,1/(p-1)\}.$
\end{theorem}

\smallskip

We should mention  that  
  Theorems~\ref{th1.1}  and ~\ref{th1.2} are of some independent of interest, and they provide an immediate 
  proof of weighted $L^p$ estimates of 
the area functions $s_h$, $s_p$, $S_{P}$ and $S_{H}$  on $L^p_w(\RN), 1<p<\infty $ and $ w\in A_p  $ 
 (see Lemma~\ref{le5.1} below). In the proofs of Theorems~\ref{th1.1}  and ~\ref{th1.2}, the main tool  is that
each area integral is controlled by  $\gL$ pointwise:

\begin{eqnarray}\label{e3.2}
Tf(x) \leq C\gL(f)(x),\ \ \ x\in\RN,
\end{eqnarray}

\noindent
where $T$ is of $S_P, S_H, s_p$ and $s_H$, and $\gL$ is defined by

\begin{eqnarray}\label{e3.1} \hspace{1cm}
\gL(f)(x)=\Bigg(\iint_{\RR_{+}^{n+1}}
\bigg({t\over t+|x-y|}\bigg)^{n\mu}
|\Psi(t\sqrt{L})f(y)|^2{dydt\over t^{n+1}}\Bigg)^{1/2}, \ \ \mu>1
\end{eqnarray}

\noindent
with some $\Psi\in{\mathcal S}(\RN)$. The idea of using $\gL$ to control the area integrals is due
 to Calder\'on and Torchinsky \cite{CT} (see also
\cite{CW} and \cite{W}).  Note that  the singular integral $\gL$ does not satisfy the standard regularity condition
of a so-called Calder\'on-Zygmund operator, thus standard techniques of Calder\'on-Zugmund
theory (\cite{CW, W}) are not applicable. The lacking of smoothness of the kernel was indeed the main
obstacle   and it was overcome by
using the method developed in \cite{ CD, DM}, together with some estimates on heat kernel bounds, finite propagation speed of solutions
to the wave equations and spectral theory of non-negative self-adjoint operators.

\smallskip

The layout of the paper is as follows.  In Section 2 we recall some basic properties of
heat kernels and finite propagation speed for the wave equation, and build the necessary kernel estimates 
for functions of an operator, which is useful in 
 the proof of weak-type $(1,1)$ estimate  for the area integrals.
  In Section 3 we will prove that the area integral is controlled by $\gL$ pointwise, 
  which implies  Theorems~\ref{th1.1} and \ref{th1.2} for $p=2$, and then we employ
  the  R. Fefferman-Pipher's method   to obtain sharp estimates for the operator norm of  
the area integrals  on $L^p(\RN)$ as $p$ becomes large. In Section  4, we will
  give the proofs of Theorems~\ref{th1.1} and \ref{th1.2}.
 Finally, in Section 5 we will prove our Theorem~\ref{th1.4}, which gives
 the growth of the $A_p$ constant on estimates
 on the weighted $L^p$ spaces.

\smallskip

Throughout, the letter ``$c$"  and ``$C$" will denote (possibly different)
constants  that are independent of the essential variables.

 \bigskip

\section{Notation and preliminaries}
\setcounter{equation}{0}

\medskip
 Let us
recall that,  if $L$ is a self-adjoint positive definite operator acting
on $L^2({\mathbb R}^n)$, then it admits a spectral resolution

\begin{eqnarray*}
L=\int_0^{\infty} \lambda dE(\lambda).
\end{eqnarray*}

 \noindent
For every bounded
Borel function $F:[0,\infty)\to{\mathbb{C}}$,
by using the spectral theorem we can define the operator

\begin{eqnarray}\label{e2.1}
F(L):=\int_0^{\infty}F(\lambda)\,dE_{L}(\lambda).
\end{eqnarray}
This is of course, bounded on $L^2({\mathbb R}^n)$.
In particular, the operator $\cos(t\sqrt{L})$ is then well-defined and bounded
on $L^2({\mathbb R}^{n})$. Moreover, it follows from Theorem 3 of \cite{CS}
  that if  the corresponding
 heat kernels $p_{t}(x,y)$ of $e^{-tL}$ satisfy Gaussian bounds $(GE)$, then  there exists a finite,
positive constant $c_0$ with the property that the Schwartz
kernel $K_{\cos(t\sqrt{L})}$ of $\cos(t\sqrt{L})$ satisfies
\begin{eqnarray}\label{e2.2} \hspace{1cm}
{\rm supp}K_{\cos(t\sqrt{L})}\subseteq
\big\{(x,y)\in {\mathbb R}^{n}\times {\mathbb R}^{n}: |x-y|\leq c_0 t\big\}.
\end{eqnarray}
\noindent
See also \cite{CGT} and  \cite{Si}. The precise value of $c_0$ is inessential and throughout the article we will choose $c_0=1$.

  By the Fourier inversion
formula, whenever $F$ is an even, bounded, Borel function with its Fourier transform
$\hat{F}\in L^1(\mathbb{R})$, we can write $F(\sqrt{L})$ in terms of
$\cos(t\sqrt{L})$. More specifically,  we have
\begin{eqnarray}\label{e2.3}
F(\sqrt{L})=(2\pi)^{-1}\int_{-\infty}^{\infty}{\hat F}(t)\cos(t\sqrt{L})\,dt,
\end{eqnarray}
which, when combined with (\ref{e2.2}), gives
\begin{eqnarray}\label{e2.4} \hspace{1cm}
K_{F(\sqrt{L})}(x,y)=(2\pi)^{-1}\int_{|t|\geq  |x-y|}{\hat F}(t)
K_{\cos(t\sqrt{L})}(x,y)\,dt,\qquad \forall\,x,y\in{\mathbb R}^{n}.
\end{eqnarray}

The following  result
 is useful for certain estimates later.

\begin{lemma}\label{le2.1}\,  Let $\varphi\in C^{\infty}_0(\mathbb R)$ be
even, $\mbox{supp}\,\varphi \subset (-1, 1)$. Let $\Phi$ denote the Fourier transform of
$\varphi$. Then for every $\kappa=0,1,2,\dots$, and for every $t>0$,
the kernel $K_{(t^2L)^{\kappa}\Phi(t\sqrt{L})}(x,y)$ of the operator
$(t^2L)^{\kappa}\Phi(t\sqrt{L})$ which was defined by the spectral theory, satisfies

\begin{eqnarray}\label{e2.5}
{\rm supp}\ \! K_{(t^2L)^{\kappa}\Phi(t\sqrt{L})}
\subseteq \big\{(x,y)\in \RN\times \RN: |x-y|\leq t\big\}
\end{eqnarray}

 \noindent
 and

\begin{eqnarray}\label{e2.6}
\big|K_{(t^2L)^{\kappa}\Phi(t\sqrt{L})}(x,y)\big|
\leq C  \, t^{-n}
\end{eqnarray}

\noindent
for all $t>0$ and $x,y\in \RN.$
\end{lemma}

\medskip

\begin{proof}  The proof of this lemma is standard (see   \cite{SW} and \cite{HLMMY}). We give a brief argument
of this proof for completeness and convenience for the reader.

For every  $\kappa=0,1,2,\dots$, we set
$\Psi_{\kappa, t}(\zeta):=(t\zeta)^{2\kappa}\Phi(t\zeta)$.
Using the definition of the Fourier transform, it can be verified that
$$
\widehat{\Psi_{\kappa,t}}(s)=(-1)^{\kappa}
{1\over t}\psi_{\kappa}({s\over t}),
$$
where we have set
$\psi_{\kappa} (s)={d^{2\kappa}\over ds^{2\kappa}}\varphi(s)$.
Observe that for every  $\kappa=0,1,2,\dots$, the function
$\Psi_{\kappa,t}\in{\mathcal S}(\mathbb R)$ is an even function.
It follows from formula (\ref{e2.4}) that

\begin{equation}\label{e2.7}
K_{(t^2L)^{\kappa}\Phi(t\sqrt{L})}(x,y)
=(-1)^{\kappa}{1\over 2\pi}\int_{|st|\geq  |x-y|}
{d^{2\kappa }\over ds^{2\kappa}}\varphi({s})K_{\cos(st\sqrt{L})}(x,y)\,ds.
\end{equation}

\noindent Since $\varphi\in C^{\infty}_0(\mathbb R)$ and
$\mbox{supp}\,\varphi \subset(-1, 1)$,  (\ref{e2.5})
  follows readily from this.

  Note that for any $m\in {\Bbb N}$ and $t>0$, we have   the relationship

  $$
  (I+tL)^{-m}={1\over  (m-1)!} \int\limits_{0}^{\infty}e^{-tsL}e^{-s} s^{m-1} ds
  $$

  \noindent
  and so when $m>n/4$,

 \begin{eqnarray*}
 \big\| (I+tL)^{-m} \big\|_{L^2\rightarrow L^{\infty}}\leq {1\over  (m-1)!} \int\limits_{0}^{\infty}
 \big\| e^{-tsL}\big\|_{L^2\rightarrow L^{\infty}} e^{-s} s^{m-1} ds\leq C t^{-n/4}
 \end{eqnarray*}

\noindent for all $t>0.$ Now $ \big\| (I+tL)^{-m} \big\|_{L^1\rightarrow L^{2}}=\big\| (I+tL)^{-m} \big\|_{L^2\rightarrow L^{\infty}}\leq
C t^{-n/4}$, and so

 \begin{eqnarray*}
 \big\|(t^2L)^{\kappa}\Phi(t\sqrt{L})\big\|_{L^1\rightarrow L^{\infty}} \leq
 \big\| (I+t^2L)^{2m}(t^2L)^{\kappa}\Phi(t\sqrt{L}) \big\|_{L^2\rightarrow L^2}  \big\| (I+t^2L)^{-m} \big\|^2_{L^2\rightarrow L^{\infty}}.
 \end{eqnarray*}

  \noindent
  The $L^2$ operator norm of the last term is equal to the $L^{\infty}$
norm of the function $(1+t^2|s|)^{2m} (t^2|s|)^{\kappa}\Phi(t\sqrt{|s|})$ which is uniformly
bounded in $t>0$. This implies that (\ref{e2.6}) holds. The proof of this lemma is concluded.
\end{proof}

\begin{lemma}\label{le2.2}\, Let $\varphi\in C^{\infty}_0(\mathbb R)$ be
even function with $\int \varphi =1$, $\mbox{supp}\,\varphi \subset (-1/10, 1/10)$.
Let $\Phi$ denote the Fourier transform of
$\varphi$ and let $\Psi(s)=s^{2n+2}\Phi^3(s)$.  Then there exists a positive constant $C=C({n,\Phi})$
such that the kernel $K_{\Psi(t\sqrt{L})(1-\Phi(r\sqrt{L}))}(x,y)$ of $ \Psi(t\sqrt{L})(1-\Phi(r\sqrt{L}))$
satisfies

\begin{eqnarray}\label{e2.8}
\big|K_{\Psi(t\sqrt{L})(1-\Phi(r\sqrt{L}))}(x,y)\big|
\leq  C\, { r\over t^{n+1}}\Big(1+{|x-y|^2\over t^2}\Big)^{-(n+1)/2}
\end{eqnarray}

\noindent
for all $t>0, r>0$ and $x, y\in \RN$.
\end{lemma}

\begin{proof} By rescaling, it is enough to
  show that

\begin{eqnarray}\label{e2.9}
|K_{\Psi(\sqrt{L})(1-\Phi(r\sqrt{L}))}(x,y)|
\leq C{r}\big(1+{|x-y|^2}\big)^{-(n+1)/2}.
\end{eqnarray}

Let us prove (\ref{e2.9}). One writes  $\Psi(s)=\Psi_1(s)\Phi^2(s)$, where $\Psi_1(s)=s^{2n+2}\Phi(s)$. Then we have
 $\Psi(\sqrt{L})=\Psi_1(\sqrt{L})\Phi^2(\sqrt{L})$. It follows from Lemma 2.1 that $|K_{\Phi(\sqrt{L}) }(z,y)|\leq C$
 and  $K_{\Phi(\sqrt{L}) }(z,y)=0$ when $|z-y|\geq 1.$ Note that if $|z-y|\leq 1$, then
 $\big(1+|x-y|\big) \leq 2(1+|x-z|)$. Hence,

\begin{eqnarray*}
&&\hspace{-1cm}\Big|\big(1+|x-y|\big)^{n+1}K_{\Psi(\sqrt{L})(1-\Phi(r\sqrt{L}))}(x,y)\Big|\\
&&=  \big(1+|x-y|\big)^{n+1}\Big|\int_{\RN}
K_{\Psi_1(\sqrt{L})(1-\Phi(r\sqrt{L}))\Phi(\sqrt{L})}(x,z) K_{\Phi(\sqrt{L}) }(z,y) dz\Big|\\
&&\leq C\int_{\RN}\big|K_{\Psi_1(\sqrt{L})(1-\Phi(r\sqrt{L}))\Phi(\sqrt{L})}(x,z)\big|\big(1+|x-z| \big)^{n+1}dz.
\end{eqnarray*}

\noindent
By symmetry, we will be done if we show that

\begin{eqnarray}\label{e2.10}
\int_{\RN}\big|K_{\Psi_1(\sqrt{L})(1-\Phi(r\sqrt{L}))\Phi(\sqrt{L})}(x,z)\big|\big(1+|x-z| \big)^{n+1}dx\leq Cr.
\end{eqnarray}

\noindent Let $G_r(s)=\Psi_1(s)(1-\Phi(rs))$. Since $G_r(s)$ is an even function, apart from a $(2\pi)^{-1}$ factor we
can write

$$G_r(s)=\int^{+\infty}_{-\infty}\widehat{G_r}(\xi){\rm cos}(s\xi)d\xi,
$$

\noindent and by (\ref{e2.3}),

\begin{eqnarray}\label{e2.11}\Psi_1(\sqrt{L})(1-\Phi(r\sqrt{L}))\Phi(\sqrt{L})
=\int^{+\infty}_{-\infty}\widehat{G_r}(\xi){\rm cos}(\xi\sqrt{L})\Phi(\sqrt{L})d\xi.
\end{eqnarray}

\noindent
By Lemma 2.1 again, it can be   seen that $K_{{\rm cos}(\xi\sqrt{L})\Phi(\sqrt{L})}(x,z)=0$ if $|x-z|\geq 1+|\xi|.$ 
Using the unitarity of $\cos(\xi\sqrt{L})$, estimates (\ref{e2.5}) and  (\ref{e2.6}), we have 

\begin{eqnarray*} 
 \int_{\RN} \big|K_{{\rm cos}(\xi\sqrt{L})\Phi(\sqrt{L})}(x,z)\big|dx &=&
  \int_{\RN} \big|\cos(\xi\sqrt{L}) \big(K_{\Phi(\sqrt{L})}(\cdot\ , z)\big)(x)\big|dx\nonumber\\
  &\leq& (1+|\xi|)^{n/2} \big\| \cos(\xi\sqrt{L}) \big(K_{\Phi(\sqrt{L})}(\cdot\ , z)\big)\big\|_{L^2(\RN)} \nonumber\\
  &\leq& (1+|\xi|)^{n/2} \big\|  K_{\Phi(\sqrt{L})}(\cdot\ , z) \big\|_{L^2(\RN)} \nonumber\\
  &\leq& (1+|\xi|)^{n/2}.
\end{eqnarray*}

\noindent
 This, 
in combination with (\ref{e2.11}), gives

\begin{eqnarray}\label{e2.12}
{\rm LHS\ \ of \ \ } (\ref{e2.10}) \ &\leq& C\int^{+\infty}_{-\infty} |\widehat{G_r}(\xi)| \, (1+|\xi|)^{2n+1}\, d\xi\nonumber\\
&\leq& C\Big(\int^{+\infty}_{-\infty} |\widehat{G_r}(\xi)|^2 \, (1+|\xi|)^{4n+4}\, d\xi\Big)^{1/2}\nonumber\\
&\leq& C\big\|G_r\|_{W^{2n+2, \,2}(\RN)}.
\end{eqnarray}

\noindent Next we estimate the term $\big\|G_r\|_{W^{2n+2, \,2}(\RN)}$.
Note that $G_r(s)=\Psi_1(s)(1-\Phi(rs)), \Phi(0)=\widehat{\varphi}(0)=\int \varphi  =1$ and
$\Phi=\widehat{\varphi}\in \mathcal{S}(\RR),$ also $\Psi_1(s)=s^{2n+2}\Phi(s).$  We have

\begin{eqnarray}\label{e2.13} \hspace{1cm}
\|G_r\|_{L^2}^2= \int_{\RR} |\Psi_1(s)|^2|1-\Phi(rs)|^2ds
 \leq  C \|\Phi'\|^2_{L^\infty} \int_{\RR} |\Psi_1(s)|^2\, (rs)^2\,ds \leq C r^2.
\end{eqnarray}

\noindent Moreover,    observe that for any $k\in{\mathbb N}$,
$\big|{d^k\over ds^k}\big(1-\Phi(rs)\big)\big|=r^k|\Phi^{(k)}(rs)|\leq Crs^{1-k}.$
By Leibniz's rule, we obtain

\begin{eqnarray}\label{e2.14}
\Big\|{d^{2n+2}\over ds^{2n+2}}G_r(s)\Big\|_{L^2}&=&  \Big\|{d^{2n+2}\over ds^{2n+2}}\Big(\Psi_1(s)\big(1-\Phi(rs)\big)\Big)\Big\|_{L^2(\RN)}\nonumber\\
 &\leq&\sum_{m+k=2n+2}\Big\|{d^{m}\over ds^{m}}\Big(s^{2n+2}\Phi\Big) {d^{k}\over ds^{k}}\Big(1-\Phi(rs)\Big)\Big\|_{L^2(\RN)}\nonumber\\
 &\leq&Cr \sum_{m=0}^{2n+2}\Big\|s^{m-(2n+1)}{d^{m}\over ds^{m}}\Big(s^{2n+2}\Phi\Big) \Big\|_{L^2(\RN)}\nonumber\\
 &\leq&Cr.
\end{eqnarray}

\noindent From estimates (\ref{e2.13}) and  (\ref{e2.14}),
it follows that  $\big\|G_r\|_{W^{2n+2, \,2}(\RN)}\leq Cr$. This, in combination with (\ref{e2.12}),
shows that the desired estimate (\ref{e2.10}) holds, and concludes the proof of Lemma 2.2.
\end{proof}

\medskip

Finally, for $s>0$, we define 
$$
{\Bbb F}(s):=\Big\{\psi:{\Bbb C}\to{\Bbb C}\ {\rm measurable}: \ \
|\psi(z)|\leq C {|z|^s\over ({1+|z|^{2s}})}\Big\}.
$$
Then for any non-zero function $\psi\in {\Bbb F}(s)$, we have that 
$\{\int_0^{\infty}|{\psi}(t)|^2\frac{dt}{t}\}^{1/2}<\infty$.
Denote by  $\psi_t(z)=\psi(tz)$. It follows from the spectral theory 
in \cite{Yo} that for any $f\in L^2(\RN)$,
 
\begin{eqnarray}
\Big\{\int_0^{\infty}\|\psi(t\sqrt{L})f\|_{L^2(\RN)}^2{dt\over t}\Big\}^{1/2}
&=&\Big\{\int_0^{\infty}\big\langle\,\overline{ \psi}(t\sqrt{L})\, 
\psi(t\sqrt{L})f, f\big\rangle {dt\over t}\Big\}^{1/2}\nonumber\\
&=&\Big\{\big\langle \int_0^{\infty}|\psi|^2(t\sqrt{L}) {dt\over t}f, 
f\big\rangle\Big\}^{1/2}\nonumber\\
&=& \kappa \|f\|_{L^2(\RN)},
\label{e2.15}
\end{eqnarray}

\noindent where $\kappa=\big\{\int_0^{\infty}|{\psi}(t)|^2 {dt/t}\big\}^{1/2},$
an estimate which will be often used in the sequel.

 \bigskip

\section{An auxiliary    $\gL$ function}
 \setcounter{equation}{0}

 \medskip

 \subsection{The $\gL$ function}
 Let $\varphi\in C^{\infty}_0(\mathbb R)$ be
even function with $\int \varphi =1$, $\mbox{supp}\,\varphi \subset (-1/10, 1/10)$.
Let $\Phi$ denote the Fourier transform of
$\varphi$ and let $\Psi(s)=s^{2n+2}\Phi^3(s)$ (see Lemma 2.2 above).
We define the $\gL$ function by

\begin{eqnarray}\label{e3.1} \hspace{1cm}
\gL(f)(x)=\Bigg(\iint_{\RR_{+}^{n+1}}
\bigg({t\over t+|x-y|}\bigg)^{n\mu}
|\Psi(t\sqrt{L})f(y)|^2{dydt\over t^{n+1}}\Bigg)^{1/2}, \ \ \mu>1.
\end{eqnarray}

\bigskip
In this section, we will  show that the area integrals $s_p$, $s_h $,  $s_H$ and $ S_H$  
are all controlled by $\gL$ pointwise. To achieve this, we  need  some results on the  kernel estimates
of the semigroup. Firstly, we note that   the Gaussian upper bounds for $p_t(x,y)$ are further
inherited by the time derivatives of $p_{t}(x,y)$. That is, for each
$k\in{\mathbb N}$, there exist two positive constants $c_k$ and $C_k$ such
that

\begin{eqnarray}\label{e3.0}
\Big|{\partial^k \over\partial t^k} p_{t}(x,y) \Big|\leq
\frac{C_k}{  t^{n /2+k} } \exp\Big(-{|x-y|^2\over
c_k\,t}\Big)
\end{eqnarray}

\noindent for all $t>0$,  and $x, y\in {\mathbb R}^{n}$. For
the proof of (\ref{e3.0}), see   \cite{Da}
and \cite{Ou}, Theorem~6.17.

Note that in the absence of regularity on space variables of $p_t(x,y)$, estimate (\ref{e3.0})
plays an important role in our theory.

\begin{lemma} \label{le3.1} \,
Let $L$ be a non-negative self-adjoint operator such that the corresponding
 heat kernels $p_t(x,y)$ of the semigroup $e^{-tL}$ satisfy Gaussian bounds $(GE)$. Then
 for every $\kappa=0,1, ..., $
  the operator $ (t\sqrt{L})^{2\kappa} e^{-t\sqrt{L}}$   satisfies

\begin{eqnarray}\label{e3.00}\hspace{1cm}
\big|K_{(t\sqrt{L})^{2\kappa} e^{-t\sqrt{L}}}(x,y)\big|\leq C_\kappa t^{-n}\Big(1+{ |x-y|\over t}\Big)^{-(n+2\kappa+1)}, \ \ \forall t>0
\end{eqnarray}

\noindent
for almost every $x,y\in \RN.$
\end{lemma}

\begin{proof}  
The proof of (\ref{e3.00}) is simple. Indeed, the subordination formula
$$e^{-t\sqrt{L}}={1\over \sqrt{\pi}}\int^\infty_0e^{-u}u^{-1/2}e^{-{t^2\over 4u}L}du$$
allows us to estimate

\begin{eqnarray*}
\big|K_{(t\sqrt{L})^{2\kappa} e^{-t\sqrt{L}}}(x,y)\big|&\leq&
C_\kappa\int_0^{\infty}{e^{-u}\over\sqrt{u}}\Big({t^2\over u}\Big)^{-n/2}\exp \Big(-{u|x-y|^2\over c t^2}\Big)u^\kappa du\\
&\leq&C_\kappa t^{-n}\int_0^{\infty}e^{-u}u^{n/2+\kappa-1/2} \exp \Big(-{u|x-y|^2\over c t^2}\Big)\, du\\
&\leq&C_\kappa t^{-n}\Big(1+{|x-y| \over  t }\Big)^{-(n+2\kappa+1)}
\end{eqnarray*}

\noindent for every $t>0$ and almost every $x,y\in \RN$.
\end{proof}

\begin{lemma} \label{le3.2} \,
Let $L$ be a non-negative self-adjoint operator such that the corresponding
 heat kernels $p_t(x,y)$ of the semigroup $e^{-tL}$ satisfy   $(GE)$ and ($G$). Then
 for every $\kappa=0,1, ..., $
  the operator $ t^{2\kappa+1}\nabla  (L^\kappa e^{-t^2L} )$   satisfies

 \begin{eqnarray*} 
\Big|K_{t^{2\kappa+1}\nabla  (L^\kappa e^{-t^2L} )}(x,y)\Big| 
&\leq& C t^{-n} \exp\Big(-{|x-y|^2\over
c \,t^2}\Big), \ \ \ \forall t>0 
\end{eqnarray*}

\noindent
for almost every $x,y\in \RN.$
\end{lemma}

\begin{proof}  Note that $ t^{2\kappa+1}\nabla  (L^\kappa e^{-t^2L} )=  t\nabla e^{-{t^2\over 2}L} \circ (t^2L)^\kappa e^{-{t^2\over 2}L}.$ 
Using (\ref{e3.0}) and   the pointwise gradient estimate ($G$) of heat kernel $p_t(x,y)$, we have 

\begin{eqnarray*} 
\Big|K_{t^{2\kappa+1}\nabla  (L^\kappa e^{-t^2L} )}(x,y)\Big|&=&  \Big| \int_{\RN} K_{t\nabla e^{-{t^2\over 2}L}}(x,z) K_{(t^2L)^\kappa e^{-{t^2\over 2}L}} (z,y)dz\Big|\nonumber\\
&\leq& C t^{-2n}\int_{\RN}    \exp\Big(-{|x-z|^2\over
c \,t^2}\Big)   \exp\Big(-{|z-y|^2\over
c \,t^2}\Big) dz\nonumber\\
&\leq& C t^{-n} \exp\Big(-{|x-y|^2\over
c \,t^2}\Big)  
\end{eqnarray*}

\noindent for every $t>0$ and almost every $x,y\in \RN$.
\end{proof}

Now we start to prove the following Propositions~\ref{prop3.3} and ~\ref{prop3.4}.

\medskip

 \begin{prop}\label{prop3.3}    Let $L$ be a non-negative self-adjoint operator such that the corresponding
 heat kernels satisfy condition   $(GE)$.
 Then for $f\in {\mathcal S}(\RN),$ there
exists a constant $C=C_{n,\mu,\Psi}$ such that the area integral $s_p$ satisfies the pointwise estimate:

\begin{eqnarray}\label{e3.4}
s_pf(x) \leq C\gL(f)(x).
\end{eqnarray}

Estimate (\ref{e3.4}) also holds for the area integral $s_h$.
\end{prop}

 \begin{prop}\label{prop3.4}    Let $L$ be a non-negative self-adjoint operator such that the corresponding
 heat kernels satisfy conditions  $(GE)$ and $(G)$.
 Then for $f\in {\mathcal S}(\RN),$ there
exists a constant $C=C_{n,\mu,\Psi}$ such that the area integral $S_P$ satisfies the pointwise estimate:

\begin{eqnarray}\label{e3.2}
S_Pf(x) \leq C\gL(f)(x).
\end{eqnarray}

Estimate (\ref{e3.2}) also holds for the area integral $S_H$.
\end{prop}

\begin{proof} [Proofs of Propositions~\ref{prop3.3} and ~\ref{prop3.4}]\ 
Let us begin  to prove (\ref{e3.2}).  
By the spectral theory (\cite{Yo}),
for every $f\in {\mathcal S}(\RN)$ and every $\kappa \in{\mathbb N}$,

\noindent\begin{eqnarray*}
f =C_\Psi\int^\infty_0
(t^2L)^{\kappa}e^{-t^2{L}}\Psi(t\sqrt{L}) f  {dt\over t}
\end{eqnarray*}

\noindent
with $C^{-1}_{\Psi}=\int^\infty_0 t^{2\kappa}
 e^{-t^2} \Psi(t ){dt/t}$, and the integral converges in $L^2(\RN)$.
Recall the subordination formula:

$$
e^{-t\sqrt{L}} ={1\over \sqrt{\pi}}\int^\infty_0 e^{-u}  {u}^{-1/2} e^{-{t^2\over 4u}L} du.
$$

\noindent
One   writes

\begin{eqnarray} \label{e3.666}
\hspace{-1cm} s\nabla e^{-s\SL}f(y)
&=& {1\over \sqrt{\pi}}\int^\infty_0e^{-u}  {u}^{-1/2}s\nabla e^{-{s^2\over 4u}L}f(y)du \nonumber\\
&=&{ C_{\Psi}  \over \sqrt{\pi}}\int^{\infty}_{0}\int^\infty_0e^{-u}  {u}^{-1/2} st^{2\kappa}\nabla\big(L^\kappa
e^{-({s^2\over 4u}+t^2)L}\big)\Psi(t\sqrt{L})f (y){dt\, du\over t}.
\end{eqnarray}

\noindent
Fix $\kappa=[{n(\mu-1)\over 2}]+1$. 
Using   Lemma~\ref{le3.2}  and  the H\"older inequality, we  can estimate (\ref{e3.666}) as follows:

\begin{eqnarray*}
&&\hspace{-1cm}|s\nabla e^{-s\SL}f(y)|\\
&\leq&C \int^{\infty}_0\int^{\infty}_0\int_{\RN}e^{-u}u^{-1/2}st^{2\kappa}\Big({s^2\over 4u}+{t^2 }\Big)^{-1/2-\kappa-n/2}e^{-|y-z|^2/c({s^2\over 4u}+{t^2 })}
|\Psi(t\SL)f(z)|{dzdtdu\over t}\\
&\leq&C A\cdot B,
\end{eqnarray*}

\noindent
where
\begin{eqnarray*}
A^2 &= &   \int^{\infty}_0\int^{\infty}_0\int_{\RN}|\Psi(t\SL)f(z)|^2e^{-u}  {u}^{-1/2}s t^{2\kappa}
\Big({s^2\over 4u}+{t^2 }\Big)^{-1/2-\kappa-n/2}e^{-|y-z|^2/c({s^2\over 4u}+{t^2 })}
{dzdtdu\over t} \end{eqnarray*} 

\noindent
and
\begin{eqnarray*}
B^2
&=&  \int^{\infty}_0\int^{\infty}_0\int_{\RN}e^{-u}u^{-1/2}st^{2\kappa}\Big({s^2\over 4u}+{t^2 }\Big)^{-1/2-\kappa-n/2}e^{-|y-z|^2/c({s^2\over 4u}+{t^2 })}
{dzdtdu\over t}\\
&=&C \int^{\infty}_0\int^{\infty}_0\int_{0}^{\infty}e^{-u}u^{-1/2}st^{2\kappa}\Big({s^2\over 4u}+{t^2 }\Big)^{-1/2-\kappa}e^{-r^2 }
r^{n-1}{drdtdu\over t}\\
 &\leq& C \int^{\infty}_0\int^{\infty}_0 e^{-u} \, v^{2\kappa}\Big(1+{v^2 }\Big)^{-1/2-\kappa} { du dv\over v} \\
 &\leq& C.
\end{eqnarray*}

\noindent
Note that in the first  equality of the above term $B$, we have changed variables  $|y-z|\rightarrow r({s^2\over 4u}+{t^2 })^{1/2}$ and $t\rightarrow v ({s^2/u})^{1/2}$.
Hence,

\begin{eqnarray*}
&&\hspace{-0.3cm}|s\nabla e^{-s\SL}f(y)|^2\\
&&\leq C
\int^{\infty}_0\int^{\infty}_0\int_{\RN}|\Psi(t\SL)f(z)|^2e^{-u}u^{-1/2}st^{2\kappa}\Big({s^2\over 4u}
+{t^2 }\Big)^{-1/2-\kappa-n/2}e^{-|y-z|^2/c({s^2\over 4u}+{t^2 })}
{dzdtdu\over t}.
\end{eqnarray*}

\noindent Therefore, we put it  into the definition of $S_p$ to obtain 

\begin{eqnarray*}
&&\hspace{-0.6cm}S^2_P(f)(x) = \int^{\infty}_0\int_{|x-y|<s}|s\nabla e^{-s\SL}f(y)|^2{dyds\over s^{n+1}}\\
&&\leq C\iint_{\RR^{n+1}_{+}}|\Psi(t\SL)f(z)|^2\\
&&\hspace{0.2cm}\times\Bigg(\int^{\infty}_0\int^{\infty}_0
\int_{|x-y|<s}e^{-u}u^{-1/2}st^{2\kappa+n}\Big({s^2\over 4u}+{t^2 }\Big)^{-1/2-\kappa-n/2}e^{-|y-z|^2/c({s^2\over 4u}+{t^2 })}
{dydsdu\over s^{n+1}}\Bigg)
{dzdt\over t^{n+1}}.
\end{eqnarray*}

\noindent
We will be done if we show that

\begin{eqnarray}\label{e3.3}
&&
\hspace{-1.2cm}
\int^{\infty}_0\int^{\infty}_0\int_{|v-y|<s}e^{-u}u^{-1/2}st^{2\kappa+n}\Big({s^2\over 4u}+{t^2 }\Big)^{-1/2-\kappa-n/2}e^{-|y|^2/c({s^2\over 4u}+{t^2 })}
{dydsdu\over s^{n+1}}\\
&\leq& C\Big({t\over t+|v|}\Big)^{n\mu},\nonumber
\end{eqnarray}

\bigskip
\noindent
where we set    $x-z=v,$  and we will prove estimate (\ref{e3.3}) by considering the following two cases.

\medskip

\noindent
{\it Case 1.}  $|v|\leq t.$ \ In this case,  it is easy to show that

\begin{eqnarray*}
{\rm LHS\ of \ (\ref{e3.3})}  
&&\leq C \int^{\infty}_0\int^{\infty}_0 e^{-u}u^{-1/2}st^{2\kappa+n}\Big({s^2\over 4u}+{t^2 }\Big)^{-1/2-\kappa-n/2}
{ dsdu\over s }\\
&&\leq C \int^{\infty}_0\int^{\infty}_0 e^{-u}u^{-1/2}\sqrt{u}s\big({s^2}+1\big)^{-1/2-\kappa-n/2}
{ dsdu\over s }\\
&&\leq C.
\end{eqnarray*}

\noindent But $|v|\leq t$, so

$$\Big({t\over t+|v|}\Big)^{n\mu}\geq C_{n,\mu}.$$

\noindent This implies that  (\ref{e3.3}) holds  when $|v|\leq t.$

\medskip

\noindent
{\it Case 2.}   $|v|> t$. In this  case, we break the integral into two pieces:

\begin{eqnarray*}
\int^\infty_0\int^{|v|/2}_0\int_{|v-y|<s}\cdots +\int^\infty_0\int^{\infty}_{|v|/2}\int_{|v-y|<s}\cdots =:I+{\it II}.
\end{eqnarray*}

\noindent For the first term, note that $|y|\geq |v|-|v-y|>|v|/2.$ This yields

\begin{eqnarray*}
I&\leq&\int^{\infty}_0\int^{|v|/2}_0\int_{|v-y|<s}e^{-u}u^{-1/2}st^{2\kappa+n}
\Big({s^2\over 4u}+{t^2 }\Big)^{-1/2-\kappa-n/2}e^{-|v|^2/c({s^2\over 4u}+{t^2 })}
{dydsdu\over s^{n+1}}\\
&\leq&C\int^{\infty}_0\int^{\infty}_0e^{-u}u^{-1/2}st^{2\kappa+n}
\Big({s^2\over 4u}+{t^2}\Big)^{-1/2-\kappa-n/2} \Big({s^2\over 4u}+{t^2 }\Big)^{n\mu/2}/|v|^{n\mu}
{ dsdu\over s }\\
&\leq&C\Big({t\over |v|}\Big)^{n\mu}\int^{\infty}_0\int^{\infty}_0
e^{-u}u^{-1/2}st^{2\kappa+n-n\mu}\Big({s^2\over 4u}+{t^2 }\Big)^{-1/2-\kappa-n/2+n\mu/2}
{ dsdu\over s }\\
&\leq&C\Big({t\over |v|}\Big)^{n\mu}\int^{\infty}_0\int^{\infty}_0e^{-u} s \big({s^2+1}\big)^{-1/2-\kappa-n/2+n\mu/2}
{ dsdu\over s }\\
&\leq&C\Big({t\over |v|}\Big)^{n\mu},
\end{eqnarray*}

\noindent where we used condition $\kappa=[{n(\mu-1)\over2}]+1$ in the last inequality.
Since $|v|>t$,
so $I\leq  C_{n,\mu}\Big({t \over t+|v|}\Big)^{n\mu}.$

\noindent
For the term $\it{II}$, we have

\begin{eqnarray*}
\it{II}
&\leq&C\int^{\infty}_0\int_{|v|/2}^{\infty}\int_{0}^{\infty}e^{-u}
u^{-1/2}st^{2\kappa+n}\big({s^2\over 4u}+{t^2 }\big)^{-1/2-\kappa-n/2}e^{-r^2/ ({s^2\over 4u}+{t^2 })}r^{n-1}
{drdsdu\over s^{n+1}}\\
&\leq&C\int^{\infty}_0\int_{|v|/2}^{\infty} e^{-u}u^{-1/2}st^{2\kappa+n}\Big({s^2\over 4u}+{t^2 }\Big)^{-1/2-\kappa}
{ dsdu\over s^{n+1}}\\
&\leq&C\int^{\infty}_0\int_{|v|/(4\sqrt{u}t)}^{\infty} e^{-u}u^{-1/2}(\sqrt{u})^{1-n}
\Big({s^2+1}\Big)^{-1/2-\kappa}
{ dsdu\over s^{n}}\\
&\leq&C\int^{\infty}_0 e^{-u}u^{-1/2}(\sqrt{u})^{1-n}
\Big({\sqrt{u}\, t\over |v|}\Big)^{2\kappa+n}
du\\&\leq&C\Big({t\over |v|}\Big)^{2\kappa+n}\\
&\leq&
 C  \Big({t\over t+|v|}\Big)^{n\mu},
\end{eqnarray*}

\noindent since $\kappa=[{n(\mu-1)\over2}]+1$ and $|v|>t$.

From the above {\it Cases 1} and {\it  2}, we have obtained estimate (\ref{e3.3}), and then   the proof of
estimate (\ref{e3.2}) is complete. 
The similar argument as above gives estimate (\ref{e3.2}) since for the area integral $S_H$, and this completes 
the proof of Proposition~\ref{prop3.4}.

For the area functions $s_P$ and $s_h$,  we can use  a similar argument to show  Proposition~\ref{prop3.3} 
by using either estimate (\ref{e3.0}) or   Lemma~\ref{le3.1} instead of Lemma~\ref{le3.2} in the  proof
of estimate (\ref{e3.2}), and we skip it here.
\end{proof}

 \bigskip

 \subsection{Weighted $L^2$ estimate of $\gL$.}
\begin{theorem}\label{th3.5}
Let $\mu>1$. Then there exists a constant $C=C_{n,\mu,\Phi}$ such that
for all $w\geq 0$ in $L_{loc}^{1}(\RN)$ and all $f\in {\mathcal S}(\RN)$, we
have
\begin{eqnarray}\label{e3.6}
\int_{\RN}\gL (f)^2wdx\leq C\int_{\RN}|f|^2Mwdx.
\end{eqnarray}
\end{theorem}

\medskip

\begin{proof} The proof essentially follows from  \cite{CWW} and \cite{CW} for the classical area function.
Note that by Lemma~\ref{le2.1}, the kernel $K_{\Psi(t\sqrt{L})}$ of the operator
$\Psi(t\sqrt{L})$ satisfies supp $K_{\Psi(t\sqrt{L})}\subseteq \big\{ (x,y)\in \RN\times \RN: |x-y|\leq t\big\}.$
By (\ref{e3.1}), one writes

\begin{eqnarray}\label{e3.7}\hspace{1cm}
\int_{\RN}\gL (f)^2wdx
&=& \int_{\RN} \int_{\RR_{+}^{n+1}}
\bigg({t\over t+|x-y|}\bigg)^{n\mu}
|\Psi(t\sqrt{L})f(y)|^2{dydt\over t^{n+1}} dx\nonumber\\
&=&
\int_{\RR^{n+1}_{+}}|\Psi(t\sqrt{L})f(y)|^2\bigg({1\over t^n}\int_{\RN}w(x)
\Big({t\over t+|x-y|}\Big)^{n\mu}dx\bigg){dydt\over t}.
\end{eqnarray}

\noindent For $k$ an integer, set

$$A_k=\Big\{(y,t):\ 2^{k-1}<{1\over t^n}\int_{\RN}w(x)
\Big({t\over t+|x-y|}\Big)^{n\mu}dx\leq 2^k\Big\}.$$

\noindent Then

\begin{eqnarray}\label{e3.8}
{\rm RHS \ of \ (\ref{e3.7})} \ \leq \sum_{k\in \mathbb{Z}}2^k
\int_{\RR^{n+1}_{+}}|\Psi(t\sqrt{L})f(y)|^2\chi_{A_k}(y,t){dydt\over t}.
\end{eqnarray}

\noindent We note that if $(y,t)\in A_k$, then since $\mu>1$,

$$2^{k-1}\leq {1\over t^n}\int_{\RN}w(x)
\Big({t\over t+|x-y|}\Big)^{n\mu}dx\leq CMw(y).$$

\noindent Now if $|y-z|<t$, then $t+|x-y|\approx t+|x-z|$. Thus if
$|y-z|<t$ and $(y,t)\in A_k,$

$$2^{k-1}\leq {C\over t^n}\int_{\RN}w(x)
\Big({t\over t+|x-z|}\Big)^{n\mu}dx\leq CMw(z).$$

\noindent In particular, if $(y,t)\in A_k$ and $|y-z|<t$, then
$z\in E_k=\{z:\ Mw(z)\geq C2^k\}$. Now since ${\rm supp}\ \! K_{ \Psi(t\sqrt{L})}(y,z)
\subseteq \big\{(y,z)\in \RN\times \RN: |y-z|\leq t\big\}$, for
$(y,t)\in A_k$,

$$\Psi(t\sqrt{L})f(y)=\int_{|y-z|<t}K_{ \Psi(t\sqrt{L})}(y,z)f(z)dz
=\int_{\RN}K_{ \Psi(t\sqrt{L})}(y,z)f(z)\chi_{E_k}(z)dz.$$

\noindent Therefore,

\begin{eqnarray*}
{\rm RHS \ of \ (\ref{e3.7})} \ &\leq& \sum_{k\in \mathbb{Z}}2^k
\int_{\RR^{n+1}_{+}}|\Psi(t\sqrt{L})(f\chi_{E_k})(y)|^2\chi_{A_k}(y,t){dydt\over t}\\
&\leq& \sum_{k\in \mathbb{Z}}2^k
\int_{\RR^{n+1}_{+}}|\Psi(t\sqrt{L})(f\chi_{E_k})(y)|^2{dydt\over t}\\
&=& \sum_{k\in \mathbb{Z}}2^k
\int_0^{\infty}\|\Psi(t\sqrt{L})(f\chi_{E_k})\|_{L^2(\RN)}^2{ dt\over t}\\
&=& C_{\Psi} \sum_{k\in \mathbb{Z}}2^k \| f\chi_{E_k} \|_{L^2(\RN)}^{2}
\end{eqnarray*}

\noindent with $C_{\Psi}=\int^\infty_0|\Psi(t)|^2{dt/t}<\infty,$ and the last inequality follows from the
spectral theory (see \cite{Yo}). By interchanging the
order of summation and integration, we have

\begin{eqnarray*}
\int_{\RN}\gL (f)^2wdx&\leq& C \sum_{k\in \mathbb{Z}}2^k\int_{\RN}|f|^2\chi_{E_k}dx\\
&\leq&  C\int_{\RN}|f|^2\Big(\sum_{k\in \mathbb{Z}}2^k\chi_{E_k}\Big)dx\\
&\leq& C \int_{\RN}|f|^2Mwdx.
\end{eqnarray*}

\noindent This concludes the proof of the theorem.
\end{proof}

\medskip
As a consequence of Propositions~\ref{prop3.3}, ~\ref{prop3.4} and Theorem~\ref{th3.5}, we have the following 
analogy for the area function of the result of Chang, Wilson and Wolff.

\medskip

\begin{cor}\label{c3.3} Let $T$ be  of the area integrals $s_h$, $s_p$, $S_{P}$ and $S_{H}.$
 Under assumptions of  Theorems~\ref{th1.1} and ~\ref{th1.2},  there exists a constant $C$ such that
 for  all $w\geq 0$ in $L_{loc}^{1}(\RN)$ and all $f\in {\mathcal S}(\RN)$,

\begin{eqnarray*}
\int_{\RN}|Tf|^2wdx\leq C\int_{\RN}|f|^2Mwdx.
\end{eqnarray*}
\end{cor}

\bigskip

\subsection{Proof of Theorem 1.3}
 Let $T$ be  of the area functions $s_h$, $s_p$, $S_{P}$ and $S_{H}.$
For $w\in A_1$, we have $Mw(x)\leq \|w\|_{A_1}w(x)$ for a.e. $x\in \RN$.
According to  Corollary \ref{c3.3},

\begin{eqnarray*}
\int_{\RN}T(f)^2wdx\leq C\int_{\RN}|f|^2Mwdx\leq C\|w\|_{A_1}\int_{\RN}|f|^2wdx.
\end{eqnarray*}

\noindent This implies (\ref{e1.8}) holds.

For (\ref{e1.9}), we follow the
method of Cordoba and Rubio de Francia (see pages 356-357, \cite{FP}). Let $p>2$ and take $f\in L^p(\RN).$ Then from duality, we know  that
there exist some $\varphi\in L^{(p/2)'}(\RN),$ with $\varphi\geq0,$ $\|\varphi\|_{L^{(p/2)'}(\RN)}=1$, such that

$$\|Tf\|^2_{L^p(\RN)}\leq \int_{\RN}|Tf|^2\varphi dx.$$

\noindent Set

$$v=\varphi +{M\varphi\over 2\|M\|_{L^{(p/2)'}(\RN)}}+{M^2\varphi\over (2\|M\|_{L^{(p/2)'}(\RN)})^2}+\cdots$$

\noindent following Rubio de Francia's familiar method  (Here $\|M\|_{L^{(p/2)'}(\RN)}$
denotes the
operator norm of the Hardy-Littlewood maximal operator on $L^{(p/2)'}(\RN)$). Then
$\|v\|_{L^{(p/2)'}(\RN)}\leq 2$ and $\|v\|_{A_1}\leq 2\|M\|_{L^{(p/2)'}(\RN)}\equiv O(p)$ as $p\to \infty.$
Therefore

\begin{eqnarray*}
\|Tf\|^2_{L^p(\RN)}&\leq& \int_{\RN}|Tf|^2\varphi dx\\
&\leq&\int_{\RN}|Tf|^2v dx\\
&\leq&C\|v\|_{A_1}\int_{\RN}|f|^2v dx\\
&\leq&Cp\|f\|^2_{L^p(\RN)}.
\end{eqnarray*}

\noindent This proves (\ref{e1.9}), and then  the proof of this theorem is complete.
\hfill{}$\Box$

\medskip

Note that in Theorem~\ref{e1.3},  when $L=-\Delta$ is the Laplacian on $\RN$, it is well known 
that  estimate (\ref{e1.9})  
of the classical area integral on $L^p(\RN)$ is sharp, in general (see, e.g., \cite{FP}).

\bigskip

\section{Proofs of  Theorems 1.1 and 1.2}
 \setcounter{equation}{0}

\medskip

Note that from Propositions~\ref{prop3.3} and ~\ref{prop3.4}, the area functions $S_H, S_P, s_H$ and $s_p$ are all controlled by the $\gL$ function.
In order to prove  Theorems 1.1 and 1.2, it suffices to show the following result.

 \medskip

\begin{theorem}\label{th4.1} \    Let $L$ be a non-negative self-adjoint operator such that the corresponding
 heat kernels satisfy Gaussian bounds $(GE)$. Let $\mu>3$.
If $w\geq 0$, $w\in L_{\rm loc}^{1}(\RN)$  and  $f\in {\mathcal S}(\RN)$, then

\begin{eqnarray*}
\hspace{-2.5cm}&{\rm (a)}& \hspace{0.1cm}\int_{\{\gL (f)>\lambda\}} wdx\leq {c(n)\over \lambda}\int_{\RN}|f| Mwdx,\ \ \ \lambda>0,\\
\hspace{-2.5cm}
&{\rm (b)}& \hspace{0.1cm}
\int_{\RN}\gL (f)^pw\, dx\leq c(n,p)\int_{\RN}|f|^pMw\, dx,\ \ \ 1<p\leq 2,
\\
\hspace{-2.5cm}&{\rm (c)}& \hspace{0.1cm} \int_{\RN} \gL (f)^pwdx\leq c(n,p)\int_{\RN}|f|^p(Mw)^{p/2}w^{-(p/2-1)}dx, \ \ \ 2<p<\infty.
\end{eqnarray*}
 \end{theorem}

 \medskip

\subsection{Weak-type  $(1,1)$ estimate}\
We first state a Whitney decomposition. For its proof, we refer to Chapter 6, \cite{St}.

\medskip

\begin{lemma} \label{le4.2}
Let $F$ be a non-empty closed set in $\RN$. Then its complement
$\Omega$ is the union of a sequence of cubes $Q_k$, whose sides are parallel to the axes,
whose interiors are mutually disjoint, and whose diameters are
approximately proportional to their distances from $F$. More explicitly:

\medskip
\noindent
(i) $\Omega=\RN\setminus F=\bigcup\limits_{k=1}^{\infty}Q_k.$

\medskip
\noindent
(ii) $Q_j\bigcap Q_k=\varnothing$ if $j\neq k$.

\medskip
\noindent
(iii)   There exist two constants $c_1, c_2 > 0$, (we can take $c_1 = 1$, and
$c_2 = 4$), so that

$$c_1 {\rm diam}(Q_k)\leq  {\rm dist}(Q_k,\  F)\leq c_2 {\rm diam}(Q_k).$$
\end{lemma}

\medskip

Note that if  $\Omega$ is an open set with  $\Omega=\bigcup\limits_{k=1}^{\infty}Q_k$
  a Whitney decomposition, then  for every $\varepsilon:\ 0<\varepsilon<1/4,$ there exists $N\in \mathbb{N}$
such that no point in $\Omega$ belongs to more than $N$ of the cubes $Q_{k}^{\ast}$, where
$Q_{k}^{\ast}=(1+\varepsilon)Q_k.$

 \bigskip

\begin{proof}[Proof of (a) of Theorem~\ref{th4.1}] 
Since $g^{\ast}_{\mu',\Psi}(f)\leq \gL(f)$ whenever $\mu'\geq \mu$, it is enough
to prove ($a$) of Theorem~\ref{th4.1}  for $3<\mu<4.$ Since $\gL$ is subadditive, we may assume that $f\geq0$ 
in the proof (if not we only need to consider the positive part and the negative part of $f$).  

For $\lambda>0,$ we set $\Omega=\{x\in\RN:\ Mf(x)>\lambda\}$. By \cite{FS} it follows that

\begin{eqnarray}\label{e4.2}
\int_{\Omega}wdx\leq {C\over \lambda}\int_{\RN}|f|Mwdx.
\end{eqnarray}

\noindent Let $\Omega=\cup Q_j$ be a Whitney decomposition, and define

\begin{eqnarray*}
h(x)&=&\left\{
         \begin{array}{ll}
           f(x), &  x\notin \Omega \\  [12pt]
           {1\over |Q_j|}\int_{Q_j}f(x)dx, & x\in Q_j
         \end{array}
       \right.\\[12pt]
b_j(x)&=&\left\{
         \begin{array}{ll}
           f(x)-{1\over |Q_j|}\int_{Q_j}f(x)dx, &  x\in Q_j \\  [12pt]
           0, & x\notin Q_j.
         \end{array}
       \right.
\end{eqnarray*}

\noindent Then $f=h+\sum_jb_j$, and we set $b=\sum_jb_j.$ As in \cite{St}, we have
$|h|\leq C\lambda$ a.e. By (\ref{e4.2}), it suffices  to show

\begin{eqnarray}\label{e4.3}
w\{x\notin\Omega:\ \gL(f)(x)>\lambda\}\leq {C\over \lambda}\int_{\RN}|f|Mwdx.
\end{eqnarray}

\noindent By Chebychev's inequality and Theorem~\ref{th3.5},

\begin{eqnarray*}
w\{x\notin\Omega:\ \gL(h)(x)>\lambda\}&\leq& {1\over \lambda^2}
\int_{\RN}\gL(h)^2(w\chi_{\RN\setminus\Omega})dx\\
&\leq&  {C\over \lambda^2}
\int_{\RN}|h|^2 M(w\chi_{\RN\setminus\Omega})dx\\
&\leq&{C\over \lambda }
\int_{\RN}|h|  M(w\chi_{\RN\setminus\Omega})dx
\end{eqnarray*}

\noindent since $|h|\leq C\lambda$ a.e. By definition of $h$, the last expression is at most

\begin{eqnarray}\label{e4.4}
{C\over \lambda }
\int_{\RN}|f|  Mwdx +\sum_j{C\over \lambda }
\int_{Q_j}\Big({1\over |Q_j|}\int_{Q_j}|f(z)|dz\Big)  M(w\chi_{\RN\setminus\Omega})(x)dx.
\end{eqnarray}

\noindent From the property (iii) of Lemma \ref{le4.2}, we know that for $x,z\in Q_j$ there is a constant
$C$ depending only on $n$ so that $M(w\chi_{\RN\setminus\Omega})(x)\leq CM(w\chi_{\RN\setminus\Omega})(z)$.
Thus (\ref{e4.4}) is less than

$${C\over \lambda }
\int_{\RN}|f|  Mwdx +\sum_j{C\over \lambda }
\int_{Q_j}\Big({1\over |Q_j|}\int_{Q_j}|f(z)|Mw(z)dz\Big) dx\leq {C\over \lambda }
\int_{\RN}|f|  Mwdx.$$

\noindent This gives

$$w\{x\notin\Omega:\ \gL(h)(x)>\lambda\}\leq{C\over \lambda }
\int_{\RN}|f|  Mwdx.$$

\noindent Therefore, estimate (\ref{e4.3}) will follow if we  show that

\begin{eqnarray}\label{eb}
w\{x\notin\Omega:\ \gL(b)(x)>\lambda\}\leq{C\over \lambda }
\int_{\RN}|f|  Mwdx.
\end{eqnarray}

\noindent  To prove (\ref{eb}),  we follow an idea of \cite{DM} to decompose  $b=\sum_jb_j=\sum_j\Phi_j(\sqrt{L})b_j + \sum_j\big(1-\Phi_j(\sqrt{L})\big)b_j,$ 
where $\Phi_j(\sqrt{L})=\Phi\Big({\ell(Q_j)\over32}\sqrt{L}\Big),$ $\Phi$ is the function as in Lemma~\ref{le2.2}
and $\ell(Q_j)$ is the side length of the cube $Q_j.$ See also \cite{CD}.
So, it reduces  to
show that

\begin{eqnarray}\label{ei1}
 w\{x\notin\Omega:\ \gL\Big(\sum_j\Phi_j(\sqrt{L})b_j\Big)(x)>\lambda\}\leq{C\over \lambda }
\int_{\RN}|f|  Mwdx
\end{eqnarray}

 \noindent
 and

\begin{eqnarray}\label{ei2}
 w\{x\notin\Omega:\ \gL\Big(\sum_j\big(1-\Phi_j(\sqrt{L})\big)b_j\Big)(x)>\lambda\}\leq{C\over \lambda }
\int_{\RN}|f|  Mwdx.
\end{eqnarray}

\noindent     By Chebychev's inequality   and Theorem \ref{th3.5} again, we have

\begin{eqnarray*}
{\rm LHS\ \ of \ \ (\ref{ei1})}\ &\leq&{C\over \lambda^2}\int_{\RN}\Big|\gL\Big(\sum_j\Phi_j(\sqrt{L})b_j\Big)\Big|^2
(w\chi_{\RN\setminus\Omega})dx\\
&\leq&{C\over \lambda^2}\int_{\RN}\Big| \sum_j\Phi_j(\sqrt{L})b_j \Big|^2
M(w\chi_{\RN\setminus\Omega})dx.
\end{eqnarray*}

\noindent Note that $\Phi_j(\sqrt{L})=\Phi\Big({\ell(Q_j)\over32}\sqrt{L}\Big),$ it follows from Lemma~\ref{le2.1}
that ${\rm supp}\  \Phi_j(\sqrt{L})b_j\subset {{17 }Q_j/16} $ and $\big|K_{\Phi_j(\sqrt{L})}(x,y)\big|\leq C/\ell(Q_j)$.
Hence, the above inequality is at most

\begin{eqnarray*} 
 {C\over \lambda^2}\sum_j \int_{\RN}\Big| \Phi_j(\sqrt{L})b_j \Big|^2
M(w\chi_{\RN\setminus\Omega})dx.
\end{eqnarray*}

\noindent
This, together with Lemma~\ref{le2.1} and the definition of $b$, 
yields

\begin{eqnarray*} 
{\rm LHS\ \ of \ \ (\ref{ei1})}\ 
&\leq& {C\over \lambda^2}\sum_j\int_{{17 }Q_j/16}\Big({\ell(Q_j)^{-n}}\int_{Q_j} |b(y)|dy \Big)^2
M(w\chi_{\RN\setminus\Omega})(x)dx\\
&\leq& {C\over \lambda^2}\sum_j\int_{{17 }Q_j/16}\Big( {1\over|Q_j|}\int_{Q_j}|f(y)|  dy \Big)^2
M(w\chi_{\RN\setminus\Omega})(x)dx\\
&\leq& {C\over \lambda }\sum_j\int_{{17 }Q_j/16}\Big( {1\over|Q_j|}\int_{ Q_j}|f(y)|  dy \Big)
M(w\chi_{\RN\setminus\Omega})(x)dx\\
&\leq& {C\over \lambda }\sum_j{1\over|Q_j|}\int_{{17 }Q_j/16} \int_{ Q_j}|f(y)|M(w\chi_{\RN\setminus\Omega})(y)  dy
dx\\
&\leq&{C\over \lambda }\int_{\RN}|f |Mwdy.
\end{eqnarray*}

\noindent This proves the desired estimate  (\ref{ei1}). 

Next we turn to  estimate (\ref{ei2}). It suffices to show that 

$$\sum_j\int_{\RN\setminus\Omega}\gL\Big( \big(1-\Phi_j(\sqrt{L})\big)b_j\Big)wdx
\leq C\int_{\RN}|f|Mwdx.$$

\noindent
Further, the above inequality reduces to prove the following result:

\begin{eqnarray}\label{ei22}
\int_{\RN\setminus\Omega}\gL\Big( \big(1-\Phi_j(\sqrt{L})\big)b_j\Big)wdx
\leq C\int_{Q_j}|f|Mwdx.
\end{eqnarray}

\noindent Let $x_j$ denote the center of $Q_j$. Let us   estimate
$\Psi(t\sqrt{L})\big(1-\Phi_j(\sqrt{L})\big)b_j(y)=:\Psi_{jt}(\sqrt{L})b_j(y)$ by considering 
 two cases: $t\leq \ell(Q_j)/4$ and $t>\ell(Q_j)/4$.

 \medskip
 
\noindent 
{\it Case 1. $t\leq \ell(Q_j)/4$}. \ In this case, we use Lemma \ref{le2.1} to obtain

\begin{eqnarray*}
 \big|\Psi_{jt}(\sqrt{L})b_j(y)\big|&\leq& |\Psi(t\sqrt{L})b_j(y)|+\Big|\Psi(t\sqrt{L})\Phi_j(\sqrt{L})b_j(y)\Big|\\
&\leq& \Big|\int_{Q_j}K_{\Psi(t\sqrt{L})}(y,z)b(z)dz\Big|\\
&&+\Big|\int_{{17\over16}Q_j}
K_{\Psi(t\sqrt{L})}(y,z)\Big(\int_{Q_j}K_{\Phi_j(\sqrt{L})}(z,x)b(x)dx\Big)dz\Big|\\
&\leq& C\|b_j\|_{ 1}t^{-n}.
\end{eqnarray*}

 \medskip
 
\noindent 
{\it Case 2. $t> \ell(Q_j)/4$}. Using  Lemma \ref{le2.2}, we have

\begin{eqnarray*}
\big|\Psi_{jt}(\sqrt{L})b_j(y)\big|
& \leq&\int_{\RN}\Big|K_{\Psi(t\sqrt{L})\big(1-\Phi_j(\sqrt{L})\big)}(y,z)\Big||b_j(z)|dz\\
& \leq& C \|b_j\|_{ 1}\ell(Q_j)t^{-n-1}.
\end{eqnarray*}

From the property (iii) of Lemma \ref{le4.2}, we know that if $x\notin\Omega,$
then $|x-x_j|> (\sqrt{n}+1/2)\ell(Q_j)$. By Lemma \ref{le2.1}, we have
$\Psi(t\sqrt{L})\big(1-\Phi_j(\sqrt{L})\big)b_j(y)=0$ unless $|y-x_j|\leq t+(1/32+\sqrt{n}/2)\ell(Q_j).$
Note that for $x\notin\Omega$, $0<t\leq \ell(Q_j)/4$ and $|y-x_j|\leq t+(1/32+\sqrt{n}/2)\ell(Q_j)$,
$|x-y|\geq|x-x_j|-|y-x_j|>{\sqrt{n}/2+7/32\over \sqrt{n}+1/2}|x-x_j|$. 
Denote $F_j=:\{y:\ |y-x_j|<(9/32+\sqrt{n}/2)\ell(Q_j)\}$. 
Then  for $x\notin\Omega$ and $\mu>3$, we have

\begin{eqnarray*}
&&\hspace{-1cm}\bigg(\int^{\ell(Q_j)/4}_{0}\int_{F_j}
\big|\Psi_{jt}(\sqrt{L})b_j(y)\big|^2
\Big({t\over t+|x-y|}\Big)^{n\mu}{dydt\over t^{n+1}}\bigg)^{1/2}\\
&&\leq C {\|b_j\|_1\ell(Q_j)^{n/2}\over |x-x_j|^{n\mu/2}}
\bigg(\int^{\ell(Q_j)/4}_{0}
t^{n\mu-2n-n-1}dt\bigg)^{1/2}\\
&&\leq C {\|b_j\|_1\ell(Q_j)^{n/2}\over |x-x_j|^{n\mu/2}}\ell(Q_j)^{(n\mu-3n)/2}\\
&&\leq C {\|b_j\|_1\ell(Q_j)^{-n}\over (1+|x-x_j|/\ell(Q_j))^{n\mu/2}} \\
&&\leq C |f\chi_{Q_j}|\ast \tau_{\ell(Q_j)}(x),
\end{eqnarray*}

\noindent where $\tau_{\ell(Q_j)}(x)=1/(1+|x|)^{n\mu/2}\in L^1(\RN)$.

For the next part of the integral we consider two cases: $n=1$ and $n>1$.
Note that for $x\notin\Omega$, $\ell(Q_j)/4<t\leq |x-x_j|/4$ and $y\in E_{jt}=:\{y:\ |y-x_j|\leq t+(1/32+\sqrt{n}/2)\ell(Q_j)\}$,
$|x-y|\geq|x-x_j|-|y-x_j|>{\sqrt{n}/4+11/32\over \sqrt{n}+1/2}|x-x_j|$. Thus for $3<\mu<4$, if $n=1$,

\begin{eqnarray*}
&&\hspace{-1cm}\bigg(\int_{\ell(Q_j)/4}^{|x-x_j|/4}\int_{E_{jt}}
\big|\Psi_{jt}(\sqrt{L})b_j(y)\big|^2
\Big({t\over t+|x-y|}\Big)^{n\mu}{dydt\over t^{n+1}}\bigg)^{1/2}\\[3pt]
&&\leq C {\|b_j\|_1\ell(Q_j)\over |x-x_j|^{ \mu/2}}
\bigg(\int_{\ell(Q_j)/4}^{|x-x_j|/4}
t^{-4+1+\mu-2}dt\bigg)^{1/2}\\[3pt]
&&\leq C {\|b_j\|_1\ell(Q_j) \over |x-x_j|^{ \mu/2}}\ell(Q_j)^{( \mu-4)/2}\\[3pt]
&&\leq C {\|b_j\|_1\ell(Q_j)^{-1}\over (1+|x-x_j|/\ell(Q_j))^{ \mu/2}} \\[3pt]
&&\leq C |f\chi_{Q_j}|\ast \sigma_{\ell(Q_j)}(x),
\end{eqnarray*}

\noindent where $\sigma_{\ell(Q_j)}(x)=1/(1+|x|)^{ \mu/2} $. On the other hand, for
$3<\mu<4$, if $n>1$,

\begin{eqnarray*}
&&\hspace{-1cm} \bigg(\int_{\ell(Q_j)/4}^{|x-x_j|/4}\int_{E_{jt}}
\big|\Psi_{jt}(\sqrt{L})b_j(y)\big|^2
\Big({t\over t+|x-y|}\Big)^{n\mu}{dydt\over t^{n+1}}\bigg)^{1/2}\\
&&\leq C {\|b_j\|_1\ell(Q_j)\over |x-x_j|^{ n\mu/2}}
\bigg(\int_{\ell(Q_j)/4}^{|x-x_j|/4}
t^{-2n-2+n\mu+n-n-1}dt\bigg)^{1/2}\\
&&\leq C {\|b_j\|_1\ell(Q_j) \over |x-x_j|^{ n\mu/2}}\ell(Q_j)^{( n\mu-2n-2)/2}\\
&&\leq C {\|b_j\|_1\ell(Q_j)^{-n}\over (1+|x-x_j|/\ell(Q_j))^{ n+1}} \\
&&\leq C |f\chi_{Q_j}|\ast P_{\ell(Q_j)}(x),
\end{eqnarray*}

\noindent where $P_{\ell(Q_j)}(x)=1/(1+|x|)^{ n+1} $. 

Finally, since $t/(t+|x-y|)\leq 1$, so

\begin{eqnarray*}
&&\hspace{-1cm}\bigg(\int^{\infty}_{|x-x_j|/4}\int_{E_{jt}}
\big|\Psi_{jt}(\sqrt{L})b_j(y)\big|^2
\Big({t\over t+|x-y|}\Big)^{n\mu}{dydt\over t^{n+1}}\bigg)^{1/2}\\
&&\leq C {\|b_j\|_1\ell(Q_j) }
\bigg(\int^{\infty}_{|x-x_j|/4}
t^{-2n-3}dt\bigg)^{1/2}\\
&&\leq C |f\chi_{Q_j}|\ast P_{\ell(Q_j)}(x).
\end{eqnarray*}

\noindent Therefore, if $x\notin\Omega$, and $n>1$, then
$\gL\Big(\big(1-\Phi_j(\sqrt{L})\big)b_j\Big)(x)\leq C |f\chi_{Q_j}|\ast P_{\ell(Q_j)}(x).$
And

\begin{eqnarray*}
\int_{\RN\setminus\Omega}\gL\Big( \big(1-\Phi_j(\sqrt{L})\big)b_j\Big)wdx
&\leq& C \int_{\RN\setminus\Omega}|f\chi_{Q_j}|\ast P_{\ell(Q_j)}wdx\\
&\leq& C \int_{Q_j}|f|(  P_{\ell(Q_j)}\ast w)dx\\
&\leq& C \int_{Q_j}|f|Mwdx.
\end{eqnarray*}

\noindent If $n=1$ we get the same thing, but with $P$ replaced by $\sigma.$ This concludes the proof of
(\ref{e4.3}). And the proof of this theorem is complete.
\end{proof}

\subsection{Estimate for $2<p<\infty$} 
We proceed by duality. If $h(x)\geq0$ and $h\in L^{(p/2)'}(wdx)$, then

\begin{eqnarray*}
\int_{\RN}\gL(f)^2hwdx=\int_{\RR^{n+1}_{+}}|\Psi(t\sqrt{L})f(y)|^2{1\over t}
\bigg({1\over t^n}\int_{\RN}h(x)w(x)\Big({t\over t+|x-y|}\Big)^{n\mu}dx\bigg)dydt.
\end{eqnarray*}

\noindent Set

$$E_k=\Big\{(y,t):\  {1\over t^n}\int_{\RN}h(x)w(x)\Big({t\over t+|x-y|}\Big)^{n\mu}dx\sim2^k\Big\}.$$

\noindent Note that if $|y-z|<t$, then $t+|x-y|\sim t+|x-z|$. Thus, if $(y,t)\in E_k$ and
$|y-z|<t$, then

\begin{eqnarray*}
2^k<{1\over t^n}\int_{\RN}h(x)w(x)\Big({t\over t+|x-y|}\Big)^{n\mu}dx\sim
{1\over t^n}\int_{\RN}h(x)w(x)\Big({t\over t+|x-z|}\Big)^{n\mu}dx.
\end{eqnarray*}

\noindent The last expression is at most

\begin{eqnarray*}
&&\hspace{-1.2cm}C\sum_{j=0}^{\infty}{1\over2^{jn\mu}}{1\over t^n}\int_{B(z,2^jt)}hwdx\\
&=&C\sum_{j=0}^{\infty}{1\over2^{jn(\mu-1)}}
{w(B(z,2^jt))\over (2^jt)^n}{1\over w(B(z,2^jt))}\int_{B(z,2^jt)}hwdx\\
&\leq&C\sum_{j=0}^{\infty}{1\over2^{jn(\mu-1)}}Mw(z)M_w(h)(z)\\
&\leq&C Mw(z)M_w(h)(z),
\end{eqnarray*}

\noindent where

$$M_w(h)(z)=\sup_{t>0}\Big({1\over w(B(z,t))}\int_{B(z,t)}hwdx\Big).$$
Recall that ${\rm supp}\ K_{\Psi(t\sqrt{L})}(y,\cdot)\subset B(y,t)$. Since for
$(y,t)\in E_k$ and $|y-z|<t$ we have $z\in A_k=\{z:\  Mw(z)M_w(h)(z)\geq C_{n,\mu}2^k\},$ it
follows that for $(y,t)\in E_k$,
 $\Psi(t\sqrt{L})f(y)=\Psi(t\sqrt{L})(f\chi_{A_k})(y).$ 
Thus,

\begin{eqnarray*}
&&\hspace{-1.2cm}\int_{\RR^{n+1}_{+}}|\Psi(t\sqrt{L})f(y)|^2{1\over t}
\bigg({1\over t^n}\int_{\RN}h(x)w(x)\Big({t\over t+|x-y|}\Big)^{n\mu}\bigg)dydt\\
&&\leq \sum_k 2^{k+1}\int_{E_k}|\Psi(t\sqrt{L})f(y)|^2{dydt\over t}\\
&&= \sum_k 2^{k+1}\int_{E_k}|\Psi(t\sqrt{L})(f\chi_{A_k})(y)|^2{dydt\over t}\\
&&\leq C  \sum_k 2^{k+1}\int_{\RN}| f|^2\chi_{A_k}{dy }\\
&&\leq C   \int_{\RN}| f|^2 {MwM_w(h) }{dy }.
\end{eqnarray*}

\noindent Applying the H\"older  inequality with exponents $p/2$ and $(p/2)'$, we obtain the bound

$$ C   \bigg(\int_{\RN}| f|^p (Mw)^{p/2}w^{-(p/2-1)}{dy }\Big)^{2/p}
\Big(\int_{\RN}M_w(h)^{(p/2)'}w{dy }\bigg)^{(p-2)/p}.$$
However, since $M_w$ is the centered maximal function, we have

$$\int_{\RN}M_w(h)^{(p/2)'}wdx\leq C_{n,p}\int_{\RN}h^{(p/2)'}wdx,$$
by a standard argument based on the Besicovitch covering lemma. Since
$h$ is arbitrary, we obtain   our result.

\bigskip

\section{Proof of Theorem~\ref{th1.4} }
\setcounter{equation}{0}

\medskip

In order to prove Theorem \ref{th1.4}, we first prove  the following result.

\begin{lemma}\label{le5.1}
Let $T$ be     of the area functions $s_h$, $s_p$, $S_{P}$, $S_{H}$ and  $\gL$ with $\mu>3$.   
Under assumptions of  Theorems~\ref{th1.1}, ~\ref{th1.2} and ~\ref{th4.1}, for $w\in A_p,\
1<p<\infty$, we have

\begin{eqnarray}
\label{e5.1}\|Tf\|_{L^p_w(\RN)}\leq C \|f\|_{L^p_w(\RN)}
\end{eqnarray}

\noindent
 where
  constant $C$ depends only on
$p$, $n$ and $w$.
\end{lemma}

\begin{proof}  Let   $T$ be   
  of the area functions $s_h$, $s_p$, $S_{P}$, $S_{H}$ and  $\gL$ with $\mu>3$.
 Note that if $w\in A_1$, then $Mw\leq Cw$ a.e.  By
Theorems~\ref{th1.1}, ~\ref{th1.2} and ~\ref{th4.1},   $T$ is bounded on
 $L^p_w(\RN), 1<p<\infty,$ for any $w\in A_1,$ i.e.,

$$\|Tf\|_{L^p_w(\RN)}\leq C \|f\|_{L^p_w(\RN)}.
$$

By extrapolation theorem, these operators are all  bounded on
 $L^p_w(\RN), 1<p<\infty,$ for any $w\in A_p,$ 
  and estimate (\ref{e5.1}) holds. For the detail, we refer the reader
to pages 141-142, Theorem 7.8, \cite{D}.
\end{proof}

\medskip

Going further, we introduce some definitions. Given a weight $w$, set $w(E)=\int_E w(x)dx$. The non-increasing
rearrangement of a measurable function $f$ with respect to a weight
$w$ is defined by (cf. \cite{CR})

\begin{eqnarray*}
f^\ast_w(t)=\sup_{w(E)=t} \inf_{x\in E}|f(x)|\ \ \  \ (0<t<w(\RN)).
\end{eqnarray*}

\noindent If $w\equiv1$, we use the notation $f^\ast(t)$.

Given a measurable function $f$, the local sharp maximal function $M^\sharp_{\lambda}f$ is defined by

\begin{eqnarray*}
M^\sharp_{\lambda}f(x)=\sup_{Q\ni x}\inf_{c}\big((f-c)\chi_Q\big)^\ast(\lambda|Q|)\ \ \ (0<\lambda<1).
\end{eqnarray*}

\noindent This function was introduced by Str\"omberg \cite{S}, and 
motivated by an alternate characterization of the space $BMO$ given
by John \cite{J}.

\begin{lemma}\label{le5.4}
For any $w\in A_p$ and for any locally integrable function $f$ with
$f^\ast_w(+\infty)=0$ we have

\begin{eqnarray}\label{e5.2}
\|Mf\|_{L^p_w(\RN)}\leq C\|w\|_{A_p}^{\gamma_{p,q}}\cdot\|M^\sharp_{\lambda_n}
(|f|^q)\|^{1/q}_{L^{p/q}_w(\RN )}\ \ \ (1<p<\infty,1\leq q<\infty),
\end{eqnarray}

\noindent where $\gamma_{p,q}=\max\{1/q,1/(p-1)\}$, $C$ depends only on $p,q$ and on the
underlying dimension $n$, and $\lambda_n$ depends only on $n$.
\end{lemma}
For the proof of this lemma, see Theorem 3.1 in \cite{L1}.

\begin{prop}\label{pro5.5} Let $\gL$ be a function with  $\mu>3$ in (\ref{e3.1}). 
  Then for any $f\in C^\infty_0(\RN)$  and for all
$x\in \RN,$

$$M^\sharp_{\lambda}\big(\gL(f)^2\big)(x)\leq CMf(x)^2,$$
where $C$ depends on $\lambda,\mu,\Psi$ and $n$.
\end{prop}
\begin{proof}
Given a cube $Q$, let $T(Q)=\{(y,t):\ y\in Q,0<t<l(Q)\}$, where $\ell(Q)$ denotes the
side length of $Q$. For $(y,t)\in T(Q)$, using (\ref{e2.5}) of Lemma \ref{le2.1} we have

\begin{eqnarray}\label{e5.3}
\Psi(t\sqrt{L})f(y)=\Psi(t\sqrt{L})(f\chi_{3Q})(y).
\end{eqnarray}

Now, fix a cube $Q$ containing $x$. For any $z\in Q$ we decompose
$\gL(f)^2$ into the sum of

$$I_1(z)=\iint_{T(2Q)}|\Psi(t\sqrt{L})f(y)|^2\Big({t\over t+|z-y|}\Big)^{n\mu}{dydt\over t^{n+1}}$$
and

$$I_2(z)=\iint_{\RR^{n+1}_{+} \setminus T(2Q)}|\Psi(t\sqrt{L})f(y)|^2\Big({t\over t+|z-y|}\Big)^{n\mu}{dydt\over t^{n+1}}.$$

\noindent
From Theorem \ref{th4.1}, we know that for $\mu>3$,   $\gL(f)$ is of weak type $(1, 1)$. Then using
(\ref{e5.3}), we have

\begin{eqnarray}\label{e5.4}
(I_1)^\ast(\lambda|Q|)&\leq& \Big(\gL(f\chi_{6Q})\Big)^{\ast}(\lambda|Q|)^2\\
&\leq&\Big({C\over \lambda|Q|}\int_{6Q}|f|\Big)^2\leq CMf(x)^2.\nonumber
\end{eqnarray}

\noindent
Further, for any $z_0 \in Q$ and $(y, t)\notin T (2Q)$, by the Mean
Value Theorem,

$$(t+|z-y|)^{-n\mu}-(t+|z_0-y|)^{-n\mu}\leq C\ell(Q)(t+|z-y|)^{-n\mu-1}.$$
From this and (\ref{e5.3}), using  Lemma \ref{le2.1} again and
$\mu>3$, we have

\begin{eqnarray*}
&&\hspace{-1.2cm}|I_2(z)-I_2(z_0)|\\
&&\leq C\ell(Q)\iint_{\RR^{n+1}_{+}\setminus
T(2Q)}t^{n\mu}|\Psi(t\sqrt{L})f(y)|^2\Big({1\over
t+|z-y|}\Big)^{n\mu+1}{dydt\over t^{n+1}}\\
&&\leq C\sum^\infty_{k=1}{1\over 2^k}{1\over (2^k\ell(Q))^{n\mu}}
\iint_{T(2^{k+1}Q)\setminus
T(2^{k}Q)}t^{n\mu}|\Psi(t\sqrt{L})f(y)|^2{dydt\over t^{n+1}}\\
&&\leq C\sum^\infty_{k=1}{1\over 2^k}{|2^{k+1}Q|\over
(2^k\ell(Q))^{n\mu}}\Big(\int^{2^{k+1}\ell(Q)}_0t^{n\mu-3n-1}dt\Big)
\Big(\int_{6\cdot2^kQ}|f|\Big)^2\\
&&\leq C\sum^\infty_{k=1}{1\over 2^k}\Big({1\over
|2^{k+1}Q|}\int_{6\cdot2^kQ}|f|\Big)^2\leq CMf(x)^2.
\end{eqnarray*}
 Combining this estimate with (\ref{e5.4}) yields

 \begin{eqnarray*}
\inf_{c}\Big((\gL(f)^2-c)\chi_Q\Big)^\ast(\lambda|Q|)&\leq&
\big((I_1+I_2-I_2(z_0))\chi_Q\big)^\ast(\lambda|Q|)\\
&\leq&(I_1)^\ast(\lambda|Q|)+CMf(x)^2\\
&\leq& CMf(x)^2,
 \end{eqnarray*}
which proves the desired result.
\end{proof}

\medskip
Then we have  the following result. 

\begin{theorem}\label{th5.4}
Let $T$ be    of the area functions $s_h$, $s_p$, $S_{P}$,$S_{H}$ and $\gL$ with  $ \mu>3.$
Under assumptions of  Theorems~\ref{th1.1}, ~\ref{th1.2} and ~\ref{th4.1}, for $w\in A_p,\
1<p<\infty$, if $\|f\|_{L^p_w(\RN )}<\infty$, then

\begin{eqnarray}
\label{e5.5}
\Big(\int_{\RN}\big(M(Tf)\big)^pwdx\Big)^{1/p}\leq C\|w\|_{A_p}^{\beta_p}\Big(\int_{\RN}
\big(M( f)\big)^pwdx\Big)^{1/p},
\end{eqnarray}

\noindent
 where
$\beta_p=\max\{1/2,1/(p-1)\}$, and a constant $C$ depends only on
$p$  and $n$. 
\end{theorem}

\begin{proof}
Suppose $T=\gL.$  From Lemma~\ref{le5.1}, we know that
$\gL$ is bounded on $L^p_w(\RN ) $  when $w\in A_p$. Therefore, assuming that
$\|f\|_{L^p_w(\RN)}$ is finite, we clearly obtain that
$(\gL)^\ast_w(+\infty)=0 $.   Letting $\gL(f)$ instead of $f$ in
(\ref{e5.2}) with $q=2$ and applying Proposition \ref{pro5.5}, we get

$$\Big(\int_{\RN}\big(M(\gL(f))\big)^pwdx\Big)^{1/p}\leq C\|w\|_{A_p}^{\beta_p}\Big(\int_{\RN}
\big(M( f)\big)^pwdx\Big)^{1/p}.$$

Under assumptions of  Theorems~\ref{th1.1} and ~\ref{th1.2}, it follows that  the area functions
$s_h$, $s_p$, $S_{P}$ and $S_{H}$ are all controlled by $\gL$ pointwise.
So we have the estimate (\ref{e5.5}) for $s_h$, $s_p$, $S_{P}$ and $S_{H}$.    Then the
proof of this theorem is complete.
\end{proof}

\medskip
\begin{proof}[Proof of Theorem \ref{th1.4}] 

In \cite{B}, Buckley proved that for the Hardy-Littlewood maximal
operator,

\begin{eqnarray}\label{m}
\|M\|_{L^p_w(\RN)}\leq C\|w\|_{A_p}^{1/(p-1)}\ \ (1<p<\infty),
\end{eqnarray}

\noindent and this result is sharp.

From (\ref{m}) and Theorem~\ref{th5.4},    there exists a constant $C=C(T, n, p)$ such that
 for all $w\in A_p$,

\begin{eqnarray}\label{5.7}
\|T\|_{L^p_w(\RN)}\leq C\|w\|_{A_p}^{ {1\over p-1}+ \, \max\big\{{1\over 2},\, {1\over p-1}\big\}}  \ \ \ \ \ \  (1<p<\infty),
\end{eqnarray}

\noindent where   $T$ is of the area functions $s_h$, $s_p$, $S_{P}$ and $S_{H}.$  
 This proves Theorem \ref{th1.4}.
\end{proof}

\bigskip

\noindent
{\bf Remarks.}\ 

\smallskip

 (i) \, Note 
that  when $L=-\Delta$ is the Laplacian on $\RN$, it is well known 
that the exponents $\beta_p$ of (\ref{e5.5}) in Theorem~\ref{th5.4} is best possible,
 in general (see, e.g., Theorem 1.5, \cite{L1}). 
 
  \smallskip
 
(ii)\, For the classical
 area function $S_{\varphi}$ in (\ref{e1.1}), the result of Theorem~\ref{th1.4} was recently  improved by  A. Lerner in \cite{L2}, i.e.,
 there exists a constant $C=C(S_{\varphi}, n, p)$ such that
 for all $w\in A_p, 1<p<\infty$, 
 
 \begin{eqnarray}\label{e5.8}
\|S_{\varphi}\|_{L^p_w(\RN)}\leq C\|w\|_{A_p}^{  \, \max\big\{{1\over 2},\, {1\over p-1}\big\}},
\end{eqnarray}

\noindent
and the estimate (\ref{e5.8})
is the best possible for all $1<p<\infty.$  However, we  do not know whether one can deduce the same 
bounds  (\ref{e5.8}) for the $L^p_w$ operator norms of  the area functions $s_h$, $s_p$, $S_{P}$ and $S_{H},$ 
and they are of interest in their own right. 

Note that  sharp weighted optimal bounds for singular integrals  has  been 
studied extensively, see  for examples,    \cite{CMP, HLRSUV, LOP1, LOP2, P}  and the references therein.

 \medskip
 
(iii)\, Finally, for $f\in {\mathcal S}(\RN)$, we define the (so called vertical) Littlewood-Paley-Stein  
functions  ${\mathcal G}_P $ and $  {\mathcal G}_H$
 by

 \begin{eqnarray*} 
{\mathcal G}_P(f)(x)&=&\bigg(\int_0^{\infty}
|t\nabla_x e^{-t\sqrt{L}} f(x)|^2{  dt\over t }\bigg)^{1/2},\\
{\mathcal G}_H(f)(x)&=&\bigg(\int_0^{\infty}
|t\nabla_x e^{-t^2L} f(x)|^2 {  dt\over t }\bigg)^{1/2},  
\end{eqnarray*}

\noindent
as well as the (so-called horizontal) Littlewood-Paley-Stein  functions $  g_p$ and $g_h$
 by

 \begin{eqnarray*} 
g_p(f)(x)&=&\bigg(\int_0^{\infty}
|t\sqrt{L} e^{-t\sqrt{L}} f(x)|^2 {  dt\over t }\bigg)^{1/2},\\
g_h(f)(x)&=&\bigg(\int_0^{\infty}
|t^2L e^{-t^2L} f(x)|^2 {  dt\over t }\bigg)^{1/2}. 
\end{eqnarray*}

\medskip

One then has the analogous statement as in Theorems 1.1, 1.2, 1.3 and 1.4  replacing 
$s_p, s_h, S_P, S_H $   by $g_p, g_h, {\mathcal G}_P,
  {\mathcal G}_H$, respectively.

\vskip 1cm


\noindent

{\bf Acknowledgment}:\   The research of Lixin Yan  is  supported by      NNSF of China (Grant No.  10771221)
and   National
Science Foundation for Distinguished
Young Scholars of China (Grant No.  10925106).

\end{document}